\documentclass[final]{lms_kajino}
\usepackage{amsmath,amssymb}
\usepackage[mathscr]{euscript}
\usepackage{enumitem}
\usepackage{graphicx}

\usepackage{xcolor}
\definecolor{darkgreen}{rgb}{0,0.75,0}
\definecolor{darkred}{rgb}{0.75,0,0}
\definecolor{darkmagenta}{rgb}{0.5,0,0.5}
\usepackage[colorlinks,citecolor=darkgreen,linkcolor=darkred,urlcolor=darkmagenta]{hyperref}

\DeclareEmphSequence{\slshape,\normalfont,\slshape} %This needs to be commented out for TeX Live 2019 or older.

\newcommand{\ver}{Version of January 20, 2022} % This latter option is for official versions.

\newtheorem{theorem}{Theorem}[section]
\newtheorem{corollary}[theorem]{Corollary}
\newtheorem{lemma}[theorem]{Lemma}
\newtheorem{proposition}[theorem]{Proposition}

\newnumbered{definition}[theorem]{Definition}
\newnumbered{assumption}[theorem]{Assumption}
\newnumbered{remark}[theorem]{Remark}
\newnumbered{example}[theorem]{Example}
\newnumbered{problem}[theorem]{Problem}
\newnumbered{framework}[theorem]{Framework}
\newunnumbered{notation}[theorem]{Notation}
\newcommand{\Capa}{\operatorname{Cap}}
\newcommand{\mr}[1]{\texttt{\href{http://www.ams.org/mathscinet-getitem?mr=#1}{MR#1}}}
\newcommand{\arxiv}[1]{\texttt{\href{http://arxiv.org/abs/#1}{arXiv:#1}}}

\newcommand{\one}{\mathbf{1}} %indicator
\newcommand{\zero}{\mathbf{0}} %indicator

\newcommand{\GSC}{\mathrm{GSC}}
\newcommand{\interior}{\operatorname{int}}
\newcommand{\supp}{\operatorname{supp}}

\title[Walk dimension is greater than two for Sierpi\'{n}ski carpets]% end with percent
 {An elementary proof that walk dimension is greater than two\\for Brownian motion on Sierpi\'{n}ski carpets} % This is the full title of the paper
\author{Naotaka Kajino}

\classno{28A80, 31C25, 31E05 (primary), 35K08, 60G30, 60J60 (secondary)}

\extraline{%The following two lines (``Keywords and phrases'') should be removed in the journal version.
\emph{Keywords and phrases:} Generalized Sierpi\'{n}ski carpet, canonical Dirichlet form,
	walk dimension, sub-Gaussian heat kernel estimate, singularity of energy measure
\vspace*{4pt}

\hspace*{-2pt}The author was supported in part by JSPS KAKENHI Grant Number JP18H01123.}

\begin{document}
\maketitle

\begin{abstract}
We give an elementary self-contained proof of the fact that the walk dimension of the
Brownian motion on an \emph{arbitrary} generalized Sierpi\'{n}ski carpet is greater than
two, no proof of which in this generality had been available in the literature.
Our proof is based solely on the self-similarity and hypercubic symmetry of the associated
Dirichlet form and on several very basic pieces of functional analysis and
the theory of regular symmetric Dirichlet forms. We also present an application of this fact
to the singularity of the energy measures with respect to the canonical self-similar
measure (uniform distribution) in this case, proved first by M.\ Hino in
[\emph{Probab.\ Theory Related Fields} \textbf{132} (2005), no.\ 2, 265--290].
\end{abstract}
%%%
\section{Introduction} \label{sec:intro}
%%%
It is an established result in the field of analysis on fractals that,
on the \emph{Sierpi\'{n}ski carpet} and certain generalizations of it called
\emph{generalized Sierpi\'{n}ski carpets} (see Figure \ref{fig:GSCs} below),
there exists a canonical diffusion process $\{X_{t}\}_{t\in[0,\infty)}$
which is symmetric with respect to the canonical self-similar measure
(uniform distribution) $\mu$ and satisfies the following estimates for
its transition density (heat kernel) $p_{t}(x,y)$:
\begin{equation}\label{eq:HKEdw}
\begin{split}
\frac{c_{1}}{\mu(B(x,t^{1/d_{\mathrm{w}}}))} \exp\biggl( - \Bigl( \frac{\rho(x,y)^{d_{\mathrm{w}}}}{c_{2}t} \Bigr)^{\frac{1}{d_{\mathrm{w}}-1}}\biggr)
	&\leq p_{t}(x,y)\\
&\leq\frac{c_{3}}{\mu(B(x,t^{1/d_{\mathrm{w}}}))} \exp\biggl( - \Bigl( \frac{\rho(x,y)^{d_{\mathrm{w}}}}{c_{4}t} \Bigr)^{\frac{1}{d_{\mathrm{w}}-1}}\biggr)
\end{split}
\end{equation}
for any points $x,y$ and any $t\in(0,\infty)$, where $c_{1},c_{2},c_{3},c_{4}\in(0,\infty)$
are some constants, $\rho$ is the Euclidean metric, $B(x,s)$ denotes
the open ball of radius $s$ centered at $x$, and $d_{\mathrm{w}}\in[2,\infty)$
is a characteristic of the diffusion called its \emph{walk dimension}.
This result was obtained by M.\ T.\ Barlow and R.\ F.\ Bass in their series of papers
\cite{BB89,BB92,BB99} (see also \cite{BB90,KZ,BBK,BBKT}), its direct analog was proved
also for the Sierpi\'{n}ski gasket in \cite{BP}, for nested fractals in \cite{Kum} and
for affine nested fractals in \cite{FHK}, and it is believed for essentially all the
known examples, and has been verified for many of them, that the walk dimension $d_{\mathrm{w}}$
is \emph{strictly greater than two}. Therefore \eqref{eq:HKEdw} implies in particular
that a typical distance the diffusion travels by time $t$ is of order $t^{1/d_{\mathrm{w}}}$
and is much smaller than the order $t^{1/2}$ of such a distance for the Brownian motion on
the Euclidean spaces. The estimates \eqref{eq:HKEdw} with $d_{\mathrm{w}}>2$ are called
\emph{sub-Gaussian estimates} for this reason, and are also known to imply a number of
other anomalous features of the diffusion, one of the most important among which is the
\emph{singularity} of the associated \emph{energy measures} with respect to the
reference measure $\mu$, proved recently in \cite[Theorem 2.13-(a)]{KM};
see also \cite{Kus89,Kus93,BST,Hin05,HN} for earlier results on singularity of
energy measures for diffusions on fractals.

The main concern of this paper is the proof of the strict inequality $d_{\mathrm{w}}>2$
for an \emph{arbitrary} generalized Sierpi\'{n}ski carpet (see Framework \ref{frmwrk:GSC} and
Definition \ref{dfn:GSC} below for its definition). In fact, the existing proof of $d_{\mathrm{w}}>2$
for this case due to Barlow and Bass in \cite[Proof of Proposition 5.1-(a)]{BB99}
requires a certain extra geometric assumption on the generalized Sierpi\'{n}ski carpet
(see Remark \ref{rmk:dwSC} below), and there is no proof of it in the literature that
is applicable to \emph{any} generalized Sierpi\'{n}ski carpet although they claimed
to have one in \cite[Remarks 5.4-1.]{BB99}. The purpose of the present paper is
to give such a proof as is also elementary, self-contained and based solely
on the self-similarity and hypercubic symmetry of the associated Dirichlet form
(see Theorem \ref{thm:GSCDF} below) and on several very basic pieces of functional analysis
and the theory of regular symmetric Dirichlet forms in \cite[Section 1.4]{FOT}.
This minimality of the requirements in our method is crucial for potential future
applications; in fact, our proof has been adapted by R.\ Shimizu in his recent preprint
\cite{Shi} to show the counterpart of $d_{\mathrm{w}}>2$ for a canonical self-similar
$p$-energy form on the Sierpi\'{n}ski carpet, whose detailed properties are mostly unknown
(except in the case of $p=2$, where his energy form coincides with the Dirichlet form
of the canonical diffusion). As an important consequence of $d_{\mathrm{w}}>2$, we
also see that \cite[Theorem 2.13-(a)]{KM} applies and recovers M.\ Hino's result
in \cite[Subsection 5.2]{Hin05} that for any generalized Sierpi\'{n}ski carpet
the energy measures are singular with respect to the reference measure $\mu$.

This paper is organized as follows. In Section \ref{sec:framework-main-theorem}, we
first introduce the framework of a generalized Sierpi\'{n}ski carpet and the canonical
Dirichlet form on it, then give the precise statement of our main theorem on the strict
inequality $d_{\mathrm{w}}>2$ (Theorem \ref{thm:dwSC}) and deduce the singularity of
the associated energy measures (Corollary \ref{cor:dwSC}). Finally, we give our
elementary self-contained proof of Theorem \ref{thm:dwSC} in Section \ref{sec:proof}.
\begin{notation}
Throughout this paper, we use the following notation and conventions.
\begin{enumerate}[label=\textup{(\arabic*)},align=left,leftmargin=*,topsep=2pt,parsep=0pt,itemsep=2pt]
\item The symbols $\subset$ and $\supset$ for set inclusion
	\emph{allow} the case of the equality.
\item $\mathbb{N}:=\{n\in\mathbb{Z}\mid n>0\}$, i.e., $0\not\in\mathbb{N}$.
\item The cardinality (the number of elements) of a set $A$ is denoted by $\#A$.
\item We set $a\vee b:=\max\{a,b\}$, $a\wedge b:=\min\{a,b\}$ and $a^{+}:=a\vee 0$
%	and $a^{-}:=-(a\wedge 0)$
	for $a,b\in[-\infty,\infty]$, and we use the same notation also for
	$[-\infty,\infty]$-valued functions and equivalence classes of them. All numerical
	functions in this paper are assumed to be $[-\infty,\infty]$-valued.
\item Let $K$ be a non-empty set. We define $\one_{A}=\one_{A}^{K}\in\mathbb{R}^{K}$ for $A\subset K$ by
	$\one_{A}(x):=\one_{A}^{K}(x):=\bigl\{\begin{smallmatrix}1 & \textrm{if $x\in A$,}\\ 0 & \textrm{if $x\not\in A$,}\end{smallmatrix}$
	and set $\|u\|_{\sup}:=\|u\|_{\sup,K}:=\sup_{x\in K}|u(x)|$ for $u\colon K\to[-\infty,\infty]$.
\item Let $K$ be a topological space. The interior and closure of $A\subset K$ in $K$
	are denoted by $\interior_{K}A$ and $\overline{A}^{K}$, respectively. We set
	$\mathcal{C}(K):=\{u\mid\textrm{$u\colon K\to\mathbb{R}$, $u$ is continuous}\}$
	and $\supp_{K}[u]:=\overline{K\setminus u^{-1}(0)}^{K}$ for $u\in\mathcal{C}(K)$.
%	The Borel $\sigma$-field of $K$ is denoted by $\mathcal{B}(K)$.
\item For $d\in\mathbb{N}$, we equip $\mathbb{R}^{d}$ with the Euclidean norm denoted by $|\cdot|$
	and set $\zero_{d}:=(0)_{k=1}^{d}\in\mathbb{R}^{d}$.
\end{enumerate}
\end{notation}
%
%%%%%
\begin{figure}[b]\centering
	\includegraphics[height=98pt]{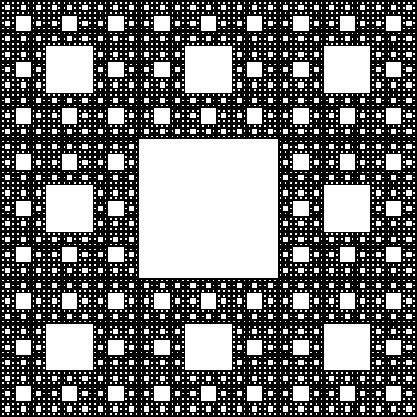}\hspace*{4pt}
	\includegraphics[height=98pt]{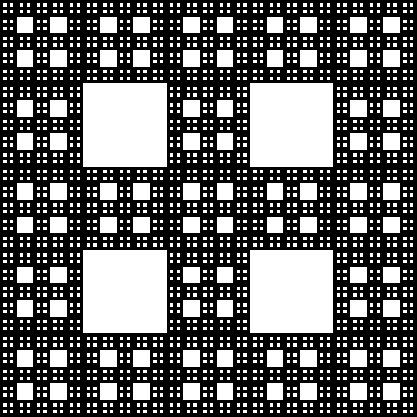}\hspace*{4pt}
	\includegraphics[height=98pt]{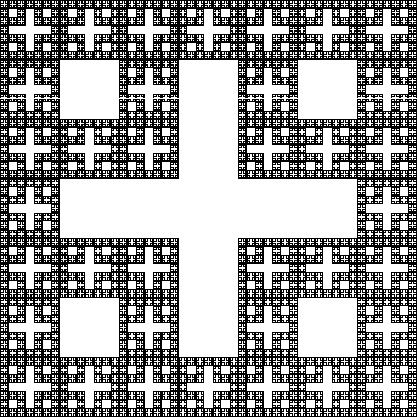}\hspace*{4pt}
	\includegraphics[height=98pt]{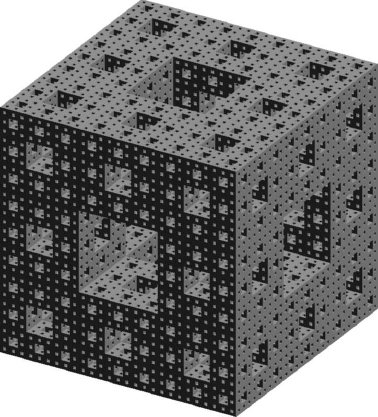}
	\caption{Sierpi\'{n}ski carpet, two other generalized Sierpi\'{n}ski
	carpets with $d=2$ and Menger sponge}\label{fig:GSCs}
\end{figure}
%%%%%
%%%
\section{Framework, the main theorem and an application}\label{sec:framework-main-theorem}
%%%
The following presentation, up to Theorem \ref{thm:GSCDF} below, is a brief summary of
the corresponding part in \cite[Section 4]{K:SPFSC}; see \cite[Section 5]{K:oscNRVNP}
and the references therein for further details.

We fix the following setting throughout this and the next sections.
\begin{framework}\label{frmwrk:GSC}
Let $d,l\in\mathbb{N}$, $d\geq 2$, $l\geq 3$ and set $Q_{0}:=[0,1]^{d}$.
Let $S\subsetneq\{0,1,\ldots,l-1\}^{d}$ be non-empty, define
$f_{i}\colon\mathbb{R}^{d}\to\mathbb{R}^{d}$ by $f_{i}(x):=l^{-1}i+l^{-1}x$ for each $i\in S$
and set $Q_{1}:=\bigcup_{i\in S}f_{i}(Q_{0})$, so that $Q_{1}\subsetneq Q_{0}$.
Let $K$ be the \emph{self-similar set} associated with $\{f_{i}\}_{i\in S}$,
i.e., the unique non-empty compact subset of $\mathbb{R}^{d}$ such that
$K=\bigcup_{i\in S}f_{i}(K)$, which exists and satisfies $K\subsetneq Q_{0}$
thanks to $Q_{1}\subsetneq Q_{0}$ by \cite[Theorem 1.1.4]{Kig01}, and
set $F_{i}:=f_{i}|_{K}$ for each $i\in S$ and $\GSC(d,l,S):=(K,S,\{F_{i}\}_{i\in S})$.
Let $\rho\colon K\times K\to[0,\infty)$ be the Euclidean metric on $K$ given by $\rho(x,y):=|x-y|$,
set $d_{\mathrm{f}}:=\log_{l}\#S$, and let $\mu$ be the \emph{self-similar measure} on
$\GSC(d,l,S)$ with weight $(1/\#S)_{i\in S}$, i.e., the unique Borel probability measure
on $K$ such that $\mu=(\#S)\mu\circ F_{i}$ (as Borel measures on $K$) for any $i\in S$,
which exists by \cite[Propositions 1.5.8, 1.4.3, 1.4.4 and Corollary 1.4.8]{Kig01}.
Let $\langle\cdot,\cdot\rangle$ and $\|\!\cdot\!\|_{2}$ denote the inner product
and norm on $L^{2}(K,\mu)$, respectively.
\end{framework}
Recall that $d_{\mathrm{f}}$ is the Hausdorff dimension of $(K,\rho)$ and that
$\mu$ is a constant multiple of the $d_{\mathrm{f}}$-dimensional Hausdorff measure
on $(K,\rho)$; see, e.g., \cite[Proposition 1.5.8 and Theorem 1.5.7]{Kig01}.
Note that $d_{\mathrm{f}}<d$ by $S\subsetneq\{0,1,\ldots,l-1\}^{d}$.

The following definition is due to Barlow and Bass \cite[Section 2]{BB99}, except that
the non-diagonality condition in \cite[Hypotheses 2.1]{BB99} has been strengthened
later in \cite{BBKT} to fill a gap in \cite[Proof of Theorem 3.19]{BB99};
see \cite[Remark 2.10-1.]{BBKT} for some more details of this correction.
\begin{definition}[(Generalized Sierpi\'{n}ski carpet, {\cite[Subsection \textup{2.2}]{BBKT}})]\label{dfn:GSC}
$\GSC(d,l,S)$ is called a \emph{generalized Sierpi\'{n}ski carpet}
if and only if the following four conditions are satisfied:
\begin{enumerate}[label=\textup{(GSC\arabic*)},align=left,leftmargin=*,topsep=2pt,parsep=0pt,itemsep=2pt]
\item\label{GSC1}(Symmetry) $f(Q_{1})=Q_{1}$ for any isometry $f$ of $\mathbb{R}^{d}$ with $f(Q_{0})=Q_{0}$.
\item\label{GSC2}(Connectedness) $Q_{1}$ is connected.
\item\label{GSC3}(Non-diagonality)
	$\interior_{\mathbb{R}^{d}}\bigl(Q_{1}\cap \prod_{k=1}^{d}[(i_{k}-\varepsilon_{k})l^{-1},(i_{k}+1)l^{-1}]\bigr)$
	is either empty or connected for any $(i_{k})_{k=1}^{d}\in\mathbb{Z}^{d}$ and
	any $(\varepsilon_{k})_{k=1}^{d}\in\{0,1\}^{d}$.
\item\label{GSC4}(Borders included) $[0,1]\times\{0\}^{d-1}\subset Q_{1}$.
\end{enumerate}
\end{definition}
As special cases of Definition \ref{dfn:GSC}, $\GSC(2,3,S_{\mathrm{SC}})$ and
$\GSC(3,3,S_{\mathrm{MS}})$ are called the \emph{Sierpi\'{n}ski carpet}
and the \emph{Menger sponge}, respectively, where
$S_{\mathrm{SC}}:=\{0,1,2\}^{2}\setminus\{(1,1)\}$ and
$S_{\mathrm{MS}}:=\bigl\{(i_{1},i_{2},i_{3})\in\{0,1,2\}^{3}\bigm|\sum_{k=1}^{3}\one_{\{1\}}(i_{k})\leq 1\bigr\}$
(see Figure \ref{fig:GSCs} above).

See \cite[Remark 2.2]{BB99} for a description of the meaning of each of the four
conditions \ref{GSC1}, \ref{GSC2}, \ref{GSC3} and \ref{GSC4} in Definition \ref{dfn:GSC}.
To be precise, \ref{GSC3} is slightly different from the formulation of the non-diagonality
condition in \cite[Subsection 2.2]{BBKT}, but they have been proved to be equivalent to
each other in \cite[Theorem 2.4]{K:NDSC}; see \cite[\S 2]{K:NDSC} for some other equivalent
formulations of the non-diagonality condition.

Throughout the rest of this paper, we assume that $\GSC(d,l,S)=(K,S,\{F_{i}\}_{i\in S})$
as introduced in Framework \ref{frmwrk:GSC} is a generalized Sierpi\'{n}ski carpet
as defined in Definition \ref{dfn:GSC}.

We next recall the result on the existence and uniqueness of a canonical diffusion
(Brownian motion) on $\GSC(d,l,S)$, which can be presented most efficiently in the language
of its associated (regular symmetric) Dirichlet form on $L^{2}(K,\mu)$ as follows;
see \cite{FOT,CF} for the basics of regular symmetric Dirichlet forms and associated
symmetric Markov processes. Below we state only the final consequence of the unique
characterizations of a canonical Dirichlet form on $\GSC(d,l,S)$ established in \cite{BBKT}
(combined with some complementary discussions in \cite{Hin13,K:oscNRVNP}), and we refer
the reader to \cite[Section 1]{BBKT} for a description of the earlier results on its
existence in \cite{BB89,KZ,BB99}.
\begin{definition}\label{dfn:GSC-isometry}
We define
\begin{equation}\label{eq:GSC-isometry}
\mathcal{G}_{0}:=\{f|_{K}\mid\textrm{$f$ is an isometry of $\mathbb{R}^{d}$, $f(Q_{0})=Q_{0}$}\},
\end{equation}
which forms a finite subgroup of the group of homeomorphisms of $K$ by virtue of \ref{GSC1}.
\end{definition}
\begin{theorem}[({\cite[Theorems \textup{1.2} and \textup{4.32}]{BBKT}}, {\cite[Proposition \textup{5.1}]{Hin13}}, {\cite[Proposition \textup{5.9}]{K:oscNRVNP}})]\label{thm:GSCDF}
There exists a unique (up to constant multiples of $\mathcal{E}$) regular symmetric Dirichlet
form $(\mathcal{E},\mathcal{F})$ on $L^{2}(K,\mu)$ satisfying $\mathcal{E}(u,u)>0$ for some
$u\in\mathcal{F}$, $\one_{K}\in\mathcal{F}$, $\mathcal{E}(\one_{K},\one_{K})=0$, and the following:
\begin{enumerate}[label=\textup{(GSCDF\arabic*)},align=left,leftmargin=*,topsep=2pt,parsep=0pt,itemsep=2pt]
\item\label{GSCDF1}If $u\in \mathcal{F}\cap\mathcal{C}(K)$ and $g\in\mathcal{G}_{0}$
	then $u\circ g\in\mathcal{F}$ and $\mathcal{E}(u\circ g,u\circ g)=\mathcal{E}(u,u)$.
\item\label{GSCDF2}$\mathcal{F}\cap\mathcal{C}(K)=\{u\in\mathcal{C}(K)\mid\textrm{$u\circ F_{i}\in\mathcal{F}$ for any $i\in S$}\}$.
\item\label{GSCDF3}There exists $r\in(0,\infty)$ such that for any $u\in\mathcal{F}\cap\mathcal{C}(K)$,
	\begin{equation}\label{eq:GSCDF3}
	\mathcal{E}(u,u)=\sum_{i\in S}\frac{1}{r}\mathcal{E}(u\circ F_{i},u\circ F_{i}).
	\end{equation}
\end{enumerate}
\end{theorem}
Throughout the rest of this paper, we fix $(\mathcal{E},\mathcal{F})$ and $r$ as given in
Theorem \ref{thm:GSCDF}; note that $r$ is uniquely determined by $(\mathcal{E},\mathcal{F})$,
since $\mathcal{E}(u,u)>0$ for some $u\in\mathcal{F}\cap\mathcal{C}(K)$ by the existence of such
$u\in\mathcal{F}$ and the denseness of $\mathcal{F}\cap\mathcal{C}(K)$ in the Hilbert space
$(\mathcal{F},\mathcal{E}_{1}:=\mathcal{E}+\langle\cdot,\cdot\rangle)$.
\begin{definition}\label{dfn:GSCDF}
The regular symmetric Dirichlet form $(\mathcal{E},\mathcal{F})$ on $L^{2}(K,\mu)$
is called the \emph{canonical Dirichlet form} on $\GSC(d,l,S)$,
and the \emph{walk dimension} $d_{\mathrm{w}}$ of $(\mathcal{E},\mathcal{F})$
(or of $\GSC(d,l,S)$) is defined by $d_{\mathrm{w}}:=\log_{l}(\#S/r)$.
%The regular symmetric Dirichlet form $(\mathcal{E},\mathcal{F})$ on $L^{2}(K,\mu)$ is called the
%\emph{canonical Dirichlet form} (or Dirichlet form of \emph{Brownian motion}) on $\GSC(d,l,S)$,
%and the \emph{walk dimension} $d_{\mathrm{w}}$ of $(\mathcal{E},\mathcal{F})$
%(or of Brownian motion on $\GSC(d,l,S)$) is defined by $d_{\mathrm{w}}:=\log_{l}(\#S/r)$.
\end{definition}
\begin{remark}\label{rmk:GSCDF}
The walk dimension $d_{\mathrm{w}}$ defined in Definition \ref{dfn:GSCDF} coincides
with the exponent $d_{\mathrm{w}}$ in \eqref{eq:HKEdw} for the regular symmetric
Dirichlet space $(K,\mu,\mathcal{E},\mathcal{F})$ equipped with the Euclidean metric $\rho$;
see the proof of Corollary \ref{cor:dwSC} below and the references therein for details.
\end{remark}
The main result of this paper is an elementary self-contained proof of the following
theorem based solely on the setting and properties stated in Framework \ref{frmwrk:GSC},
Definitions \ref{dfn:GSC}, \ref{dfn:GSC-isometry} and Theorem \ref{thm:GSCDF}
(except the uniqueness of $(\mathcal{E},\mathcal{F})$) and on several very basic pieces of
functional analysis and the theory of regular symmetric Dirichlet forms in \cite[Section 1.4]{FOT}.
To keep the whole treatment as elementary and self-contained as possible,
in our proof of Theorem \ref{thm:dwSC} we refrain from using any known properties of
$(\mathcal{E},\mathcal{F})$ other than those in Theorem \ref{thm:GSCDF}.
\begin{theorem}[(Cf.\ {\cite[Remarks 5.4-1.]{BB99}})]\label{thm:dwSC}
$d_{\mathrm{w}}>2$.
\end{theorem}
\begin{remark}\label{rmk:dwSC}
\emph{No proof of Theorem \textup{\ref{thm:dwSC}} in the present generality had been available
in the literature}, although Barlow and Bass claimed to have one in \cite[Remarks 5.4-1.]{BB99}. Its
existing proof in \cite[Proof of Proposition 5.1-(a)]{BB99} requires the extra assumption on $\GSC(d,l,S)$ that
\begin{equation}\label{eq:dwSC-BB99}
\#\{(i_{k})_{k=1}^{d}\in S\mid i_{1}=j\}\not=\#\{(i_{k})_{k=1}^{d}\in S\mid i_{1}=0\}
	\quad\textrm{for some $j\in\{1,\ldots,l-1\}$},
\end{equation}
which holds for any generalized Sierpi\'{n}ski carpet with $d=2$ but does fail
for infinitely many examples of generalized Sierpi\'{n}ski carpets with fixed $d$
for each $d\geq 3$; indeed, for each $d,l\in\mathbb{N}$ with $d\geq 3$ and $l\geq 2$,
it is not difficult to see that $\GSC(d,2ld,S_{d,l})$ with
\begin{equation}\label{eq:dwSC-BB99-counterexamples}
S_{d,l}:=\biggl\{i\biggm|
	\begin{minipage}{290pt}
	$i=(i_{k})_{k=1}^{d}\in\{0,1,\ldots,2ld-1\}^{d}$, and for any $j\in\{1,3,\ldots,2l-1\}$,
	$\{|2i_{k}-2ld+1|\mid k\in\{1,2,\ldots,d\}\}\not=\{j,j+2l,\ldots,j+2l(d-1)\}$
	\end{minipage}
	\biggr\}
\end{equation}
satisfies \ref{GSC1}, \ref{GSC2}, \ref{GSC3} and \ref{GSC4} in Definition \ref{dfn:GSC}
but not \eqref{eq:dwSC-BB99}.
\end{remark}
The proof of Theorem \ref{thm:dwSC} is given in the next section.
We conclude this section by presenting an application of Theorem \ref{thm:dwSC}
to the singularity with respect to $\mu$ of the energy measures associated with
$(K,\mu,\mathcal{E},\mathcal{F})$, which was proved first by Hino in
\cite[Subsection 5.2]{Hin05} via $d_{\mathrm{w}}>2$ and is obtained here
by combining \cite[Theorem 2.13-(a)]{KM} with $d_{\mathrm{w}}>2$.
\begin{definition}[(Cf.\ {\cite[(3.2.13), (3.2.14) and (3.2.15)]{FOT}})]\label{d:EnergyMeas}
The \emph{$\mathcal{E}$-energy measure} $\mu_{\langle u\rangle}$ of
$u\in\mathcal{F}$ is defined, first for $u\in\mathcal{F}\cap L^{\infty}(K,\mu)$
as the unique ($[0,\infty]$-valued) Borel measure on $K$ such that
\begin{equation}\label{e:EnergyMeas}
\int_{K}v\,d\mu_{\langle u\rangle}=\mathcal{E}(uv,u)-\frac{1}{2}\mathcal{E}(v,u^{2})
	\qquad\textrm{for any $v\in\mathcal{F}\cap\mathcal{C}(K)$,}
\end{equation}
and then by
$\mu_{\langle u\rangle}(A):=\lim_{n\to\infty}\mu_{\langle(-n)\vee(u\wedge n)\rangle}(A)$
for each Borel subset $A$ of $K$ for general $u\in\mathcal{F}$;
note that $uv\in\mathcal{F}$ for any $u,v\in\mathcal{F}\cap L^{\infty}(K,\mu)$ by
\cite[Theorem 1.4.2-(ii)]{FOT} and that $\{(-n)\vee(u\wedge n)\}_{n=1}^{\infty}\subset\mathcal{F}$
and $\lim_{n\to\infty}\mathcal{E}\bigl(u-(-n)\vee(u\wedge n),u-(-n)\vee(u\wedge n)\bigr)=0$
by \cite[Theorem 1.4.2-(iii)]{FOT}.
\end{definition}
\begin{corollary}[({\cite[Subsection 5.2]{Hin05}})]\label{cor:dwSC}
$\mu_{\langle u\rangle}$ is singular with respect to $\mu$ for any $u\in\mathcal{F}$.
\end{corollary}
\begin{proof}
(Unlike the proof of Theorem \ref{thm:dwSC}, this proof is \emph{not} meant to be self-contained.)
$(\mathcal{E},\mathcal{F})$ is local by \cite[Lemma 3.4]{K:cdsa}, whose proof is based
only on \ref{GSCDF2}, \ref{GSCDF3} and \cite[Exercise 1.4.1 and Theorem 3.1.2]{FOT}, and
is therefore strongly local since $\mathcal{E}(\one_{K},v)=0$ for any $v\in\mathcal{F}$
by $\mathcal{E}(\one_{K},\one_{K})=0$ (see Lemma \ref{lem:conservative} below).
We easily see that $c_{5}s^{d_{\mathrm{f}}}\leq\mu(B(x,s))\leq c_{6}s^{d_{\mathrm{f}}}$
for any $(x,s)\in K\times(0,d]$ for some $c_{5},c_{6}\in(0,\infty)$, where
$B(x,s):=\{y\in K\mid\rho(x,y)<s\}$. It is also immediate that $(K,\rho)$ satisfies the
chain condition as defined in \cite[Definition 2.10-(a)]{KM}, in view of the fact that
by \ref{GSC4}, \ref{GSC1} and \ref{GSC2} there exists $c_{7}\in(0,\infty)$ such that
for any $x,y\in K$ there exists a continuous map $\gamma\colon[0,1]\to K$ with $\gamma(0)=x$
and $\gamma(1)=y$ whose Euclidean length is at most $c_{7}\rho(x,y)$. Finally,
by \cite[Theorem 4.30 and Remark 4.33]{BBKT} (see also \cite[Theorem 1.3]{BB99})
the heat kernel $p_{t}(x,y)$ of $(K,\mu,\mathcal{E},\mathcal{F})$ exists and
there exist $\beta_{0}\in(1,\infty)$ and $c_{1},c_{2},c_{3},c_{4}\in(0,\infty)$
such that \eqref{eq:HKEdw} with $\beta_{0}$ in place of $d_{\mathrm{w}}$ holds
for $\mu$-a.e.\ $x,y\in K$ for each $t\in(0,\infty)$, but then necessarily
$\beta_{0}=\log_{l}(\#S/r)=d_{\mathrm{w}}$ by \ref{GSCDF2}, \ref{GSCDF3} and \cite[Theorem 4.31]{BBKT}
as shown in \cite[Proof of Proposition 5.9, Second paragraph]{K:oscNRVNP},
whence $\beta_{0}=d_{\mathrm{w}}>2$ by Theorem \ref{thm:dwSC}. Thus
$(K,\rho,\mu,\mathcal{E},\mathcal{F})$ satisfies all the assumptions of
\cite[Theorem 2.13-(a)]{KM}, which implies the desired claim.
\end{proof}
%
%%%
\section{The elementary proof of the main theorem}\label{sec:proof}
%%%
This section is devoted to giving our elementary self-contained proof of the main
theorem (Theorem \ref{thm:dwSC}), which is an adaptation of, and has been inspired by,
an elementary proof of the counterpart of Theorem \ref{thm:dwSC} for Sierpi\'{n}ski
gaskets presented in \cite[Proof of Proposition 5.3, Second paragraph]{KM}.
We start with basic definitions and some simple lemmas.
\begin{definition}\label{dfn:words}
We set $W_{m}:=S^{m}=\{w_{1}\ldots w_{m}\mid\textrm{$w_{i}\in S$ for $i\in\{1,\ldots,m\}$}\}$
for $m\in\mathbb{N}$ and $W_{*}:=\bigcup_{m=1}^{\infty}W_{m}$. For each
$w=w_{1}\ldots w_{m}\in W_{*}$, the unique $m\in\mathbb{N}$ with $w\in W_{m}$
is denoted by $|w|$, and we set $F_{w}:=F_{w_{1}}\circ\cdots\circ F_{w_{m}}$,
$K_{w}:=F_{w}(K)$ and $q^{w}=(q^{w}_{k})_{k=1}^{d}:=F_{w}(\zero_{d})$.
\end{definition}
\begin{lemma}\label{lem:cell-intersection-null}
If $w,v\in W_{*}$, $|w|=|v|$ and $w\not=v$, then
$\mu(F_{w}(K\setminus(0,1)^{d}))=0=\mu(K_{w}\cap K_{v})$.
\end{lemma}
\begin{proof}
This follows easily from \ref{GSC1} and the fact that $\mu$ is a Borel probability
measure on $K$ satisfying $\mu(K_{w})=(\#S)^{-|w|}$ for any $w\in W_{*}$.
\end{proof}
%
%\begin{definition}\label{dfn:GSC-V0}
%We set $V_{0}:=K\setminus(0,1)^{d}$ and
%$V_{0}^{\varepsilon}:=K\cap(\{\varepsilon\}\times\mathbb{R}^{d-1})$ for each $\varepsilon\in\{0,1\}$.
%\end{definition}
%
%\begin{lemma}\label{lem:GSC-V0}
%Let $w,v\in W_{*}$ satisfy $|w|=|v|$ and $w\not=v$. Then $K_{w}\cap K_{v}=F_{w}(V_{0})\cap F_{v}(V_{0})$.
%\end{lemma}
%%
%\begin{proof}
%This follows by $K\subset Q_{0}$ and
%$f_{w}(Q_{0})\cap f_{v}(Q_{0})=f_{w}(Q_{0}\setminus(0,1)^{d})\cap f_{v}(Q_{0}\setminus(0,1)^{d})$.
%\end{proof}
%
\begin{lemma}\label{lem:Fw-star}
Let $w\in W_{*}$. Then for any Borel measurable function $u\colon K\to[-\infty,\infty]$,
$\int_{K}|u\circ F_{w}|\,d\mu=(\#S)^{|w|}\int_{K_{w}}|u|\,d\mu$ and
$\int_{K_{w}}|u\circ F_{w}^{-1}|\,d\mu=(\#S)^{-|w|}\int_{K}|u|\,d\mu$. In particular, bounded
linear operators $F_{w}^{*},(F_{w})_{*}\colon L^{2}(K,\mu)\to L^{2}(K,\mu)$ can be defined by setting
\begin{equation}\label{eq:Fw-star}
F_{w}^{*}u:=u\circ F_{w}\qquad\textrm{and}\qquad
	(F_{w})_{*}u:=
	\begin{cases}
		u\circ F_{w}^{-1}&\textrm{on $K_{w}$,}\\
		0&\textrm{on $K\setminus K_{w}$}
	\end{cases}
\end{equation}
for each $u\in L^{2}(K,\mu)$. Moreover, $u\circ F_{w}\in\mathcal{F}$
and \eqref{eq:GSCDF3} holds for any $u\in\mathcal{F}$.
\end{lemma}
\begin{proof}
The former assertions are immediate from $\mu=(\#S)^{|w|}\mu\circ F_{w}$.
For the latter ones, let $u\in\mathcal{F}$.
Since $\mathcal{F}\cap\mathcal{C}(K)$ is dense in the Hilbert space
$(\mathcal{F},\mathcal{E}_{1}=\mathcal{E}+\langle\cdot,\cdot\rangle)$
by the regularity of $(\mathcal{E},\mathcal{F})$,
we can choose $\{u_{n}\}_{n=1}^{\infty}\subset\mathcal{F}\cap\mathcal{C}(K)$
so that $\lim_{n\to\infty}\mathcal{E}_{1}(u-u_{n},u-u_{n})=0$, and then
$\{u_{n}\circ F_{w}\}_{n=1}^{\infty}$ is a Cauchy sequence in $(\mathcal{F},\mathcal{E}_{1})$
with $\lim_{n\to\infty}\|u\circ F_{w}-u_{n}\circ F_{w}\|_{2}=0$ by \ref{GSCDF2} and
\ref{GSCDF3} and therefore has to converge to $u\circ F_{w}$ in norm in $(\mathcal{F},\mathcal{E}_{1})$.
Thus $u\circ F_{w}\in\mathcal{F}$, and \eqref{eq:GSCDF3} for $u$ follows by letting
$n\to\infty$ in \eqref{eq:GSCDF3} for $u_{n}\in\mathcal{F}\cap\mathcal{C}(K)$.
\end{proof}
\begin{lemma}\label{lem:conservative}
$\mathcal{E}(\one_{K},v)=0$ for any $v\in\mathcal{F}$.
\end{lemma}
\begin{proof}
This is immediate from the Cauchy--Schwarz inequality for $\mathcal{E}$ and
$\mathcal{E}(\one_{K},\one_{K})=0$.
\end{proof}
\begin{definition}\label{dfn:part-harmonic}
Let $U$ be a non-empty open subset of $K$.
\begin{enumerate}[label=\textup{(\arabic*)},align=left,leftmargin=*,topsep=2pt,parsep=0pt,itemsep=2pt]
\item Equipping $\mathcal{F}$ with the inner product
	$\mathcal{E}_{1}=\mathcal{E}+\langle\cdot,\cdot\rangle$, we define
	\begin{equation}\label{eq:part}
	\mathcal{C}_{U}:=\{u\in\mathcal{F}\cap\mathcal{C}(K)\mid\supp_{K}[u]\subset U\}
		\qquad\textrm{and}\qquad
		\mathcal{F}_{U}:=\overline{\mathcal{C}_{U}}^{\mathcal{F}},
	\end{equation}
	which are linear subspaces of $\mathcal{F}$, and for each $u\in\mathcal{F}$ we also set
	$u+\mathcal{C}_{U}:=\{u+v\mid v\in\mathcal{C}_{U}\}$ and $u+\mathcal{F}_{U}:=\{u+v\mid v\in\mathcal{F}_{U}\}$,
	so that $\overline{u+\mathcal{C}_{U}}^{\mathcal{F}}=u+\mathcal{F}_{U}$.
\item A function $h\in\mathcal{F}$ is said to be \emph{$\mathcal{E}$-harmonic} on $U$
	if and only if either of the following two conditions, which are easily seen
	to be equivalent to each other, holds:
	\begin{gather}\label{eq:harmonic-var}
	\mathcal{E}(h,h)=\inf\{\mathcal{E}(u,u)\mid u\in h+\mathcal{F}_{U}\},\\
	\mathcal{E}(h,v)=0\quad\textrm{for any $v\in\mathcal{C}_{U}$, or equivalently, for any $v\in\mathcal{F}_{U}$,}
	\label{eq:harmonic-weak}
	\end{gather}
	where the equivalence stated in \eqref{eq:harmonic-weak} is immediate from \eqref{eq:part}.
\end{enumerate}
\end{definition}
\begin{definition}\label{dfn:GSC-V00V01-isometry-subgroup}
\begin{enumerate}[label=\textup{(\arabic*)},align=left,leftmargin=*,topsep=2pt,parsep=0pt,itemsep=2pt]
\item We set $V_{0}^{\varepsilon}:=K\cap(\{\varepsilon\}\times\mathbb{R}^{d-1})$ for each
	$\varepsilon\in\{0,1\}$ and $U_{0}:=K\setminus(V_{0}^{0}\cup V_{0}^{1})$.
\item We fix an arbitrary $\varphi_{0}\in\mathcal{C}_{K\setminus V_{0}^{0}}$ with
	$\supp_{K}[\one_{K}-\varphi_{0}]\subset K\setminus V_{0}^{1}$, which exists by \cite[Exercise 1.4.1]{FOT};
	note that $\varphi_{0}+\mathcal{C}_{U_{0}},\varphi_{0}+\mathcal{F}_{U_{0}}$
	are independent of a particular choice of \nolinebreak$\varphi_{0}$.
\item We define $g_{\varepsilon}\in\mathcal{G}_{0}$ by $g_{\varepsilon}:=\tau_{\varepsilon}|_{K}$
	for each $\varepsilon=(\varepsilon_{k})_{k=1}^{d}\in\{0,1\}^{d}$,
	where $\tau_{\varepsilon}\colon\mathbb{R}^{d}\to\mathbb{R}^{d}$ is given by
	$\tau_{\varepsilon}((x_{k})_{k=1}^{d}):=(\varepsilon_{k}+(1-2\varepsilon_{k})x_{k})_{k=1}^{d}$,
	and define a subgroup $\mathcal{G}_{1}$ of $\mathcal{G}_{0}$ by
	\begin{equation}\label{eq:GSC-isometry-subgroup}
	\mathcal{G}_{1}:=\{g_{\varepsilon}\mid\varepsilon\in\{0\}\times\{0,1\}^{d-1}\}.
	\end{equation}
\end{enumerate}
\end{definition}
Now we proceed to the core part of the proof of Theorem \ref{thm:dwSC}.
It is divided into three propositions, proving respectively the existence of a good sequence
$\{u_{n}\}_{n=1}^{\infty}\subset\mathcal{F}\cap\mathcal{C}(K)$ converging in norm in
$(\mathcal{F},\mathcal{E}_{1})$ to $h_{0}\in\varphi_{0}+\mathcal{F}_{U_{0}}$
which is $\mathcal{E}$-harmonic on $U_{0}$ (Proposition \ref{prop:h0-approx}),
$\mathcal{E}(h_{0},h_{0})>0$ (Proposition \ref{prop:h0-nonzero})
and the \emph{non}-$\mathcal{E}$-harmonicity on $U_{0}$ of
$h_{2}:=\sum_{w\in W_{2}}(F_{w})_{*}(l^{-2}h_{0}+q^{w}_{1}\one_{K})\in h_{0}+\mathcal{F}_{U_{0}}$
(Proposition \ref{prop:h2-non-harmonic}); see Figures \ref{fig:h0U0-dfn} and \ref{fig:hm-dfn}
below for an illustration of $h_{0}$ and $h_{2}$. Then Theorem \ref{thm:dwSC} will follow
from $\mathcal{E}(h_{0},h_{0})<\mathcal{E}(h_{2},h_{2})$ and \eqref{eq:GSCDF3}
for $u\in\mathcal{F}$. While the existence of such $h_{0}$ is implied by
\cite[Theorems 7.2.1, 4.6.5, 1.5.2-(iii), A.2.6-(i), 4.1.3, 4.2.1-(ii) and Corollary 2.3.1]{FOT},
that of $\{u_{n}\}_{n=1}^{\infty}\subset\mathcal{F}\cap\mathcal{C}(K)$ as in the following proposition
cannot be obtained directly from the theory of regular symmetric Dirichlet forms in \cite{FOT,CF}.
\begin{proposition}\label{prop:h0-approx}
There exist $h_{0}\in\mathcal{F}$ and $\{u_{n}\}_{n=1}^{\infty}\subset\mathcal{F}\cap\mathcal{C}(K)$
satisfying the following:
\begin{enumerate}[label=\textup{(\arabic*)},align=left,leftmargin=*,topsep=2pt,parsep=0pt,itemsep=2pt]
\item\label{it:h0}$h_{0}$ is $\mathcal{E}$-harmonic on $U_{0}$ and $h_{0}\in\varphi_{0}+\mathcal{F}_{U_{0}}$.
	In particular, $h_{0}+\mathcal{F}_{U_{0}}=\varphi_{0}+\mathcal{F}_{U_{0}}$.
\item\label{it:un}For each $n\in\mathbb{N}$, $u_{n}\circ g=u_{n}$ for any $g\in\mathcal{G}_{1}$
%	$u_{n}+u_{n}\circ g_{1}=\one_{K}$ % $g_{1}:=g_{(1,0,\ldots,0)}$
	and $u_{n}\in\varphi_{0}+\mathcal{C}_{U_{0}}$.
\item\label{it:h0-approx}$\lim_{n\to\infty}\mathcal{E}_{1}(h_{0}-u_{n},h_{0}-u_{n})=0$.
\end{enumerate}
\end{proposition}
\begin{proof}
Recalling \eqref{eq:harmonic-var}, for each $\alpha\in[0,\infty)$ we set
\begin{equation}\label{eq:CapAlphaV00V01}
a_{\alpha}:=\inf\bigl\{\mathcal{E}(u,u)+\alpha\|u\|_{2}^{2}\bigm|u\in\varphi_{0}+\mathcal{C}_{U_{0}}\bigr\}
	=\inf\bigl\{\mathcal{E}(u,u)+\alpha\|u\|_{2}^{2}\bigm|u\in\varphi_{0}+\mathcal{F}_{U_{0}}\bigr\},
\end{equation}
where the latter equality in \eqref{eq:CapAlphaV00V01} is immediate from
$\overline{\varphi_{0}+\mathcal{C}_{U_{0}}}^{\mathcal{F}}=\varphi_{0}+\mathcal{F}_{U_{0}}$.
Then for any $\alpha\in[0,\infty)$ and any $u\in\varphi_{0}+\mathcal{C}_{U_{0}}$,
by the unit contraction operating on $(\mathcal{E},\mathcal{F})$ (see \cite[Section 1.1 and Theorem 1.4.1]{FOT})
we have $u^{+}\wedge 1\in\varphi_{0}+\mathcal{C}_{U_{0}}$ and
\begin{equation*}
\mathcal{E}(u,u)\geq\mathcal{E}(u^{+}\wedge 1,u^{+}\wedge 1)
	\geq\mathcal{E}(u^{+}\wedge 1,u^{+}\wedge 1)+\alpha\|u^{+}\wedge 1\|_{2}^{2}-\alpha
	\geq a_{\alpha}-\alpha,
\end{equation*}
and hence $a_{0}\geq a_{\alpha}-\alpha$, so that for each $n\in\mathbb{N}$
we can take $v_{n}\in\varphi_{0}+\mathcal{C}_{U_{0}}$ such that
\begin{equation}\label{eq:a0-approx-core}
\mathcal{E}(v_{n}^{+}\wedge 1,v_{n}^{+}\wedge 1)%\leq\mathcal{E}(v_{n},v_{n})
	\leq\mathcal{E}(v_{n},v_{n})+n^{-1}\|v_{n}\|_{2}^{2}
	<a_{n^{-1}}+n^{-1}\leq a_{0}+2n^{-1}.
\end{equation}
Recalling \ref{GSCDF1}, now for each $n\in\mathbb{N}$ we can define
$u_{n}\in\mathcal{F}\cap\mathcal{C}(K)$ with the properties in \ref{it:un} by
$u_{n}:=(\#\mathcal{G}_{1})^{-1}\sum_{g\in\mathcal{G}_{1}}(v_{n}^{+}\wedge 1)\circ g$
and see from the triangle inequality for $\mathcal{F}\ni u\mapsto\mathcal{E}(u,u)^{1/2}$,
$\mathcal{E}((v_{n}^{+}\wedge 1)\circ g,(v_{n}^{+}\wedge 1)\circ g)=\mathcal{E}(v_{n}^{+}\wedge 1,v_{n}^{+}\wedge 1)$
for $g\in\mathcal{G}_{1}$ and \eqref{eq:a0-approx-core} that
\begin{equation}\label{eq:a0-approx-core-sym}
\mathcal{E}(u_{n},u_{n})\leq\mathcal{E}(v_{n}^{+}\wedge 1,v_{n}^{+}\wedge 1)<a_{0}+2n^{-1}.
\end{equation}
Further, since $\|u_{n}\|_{2}\leq 1$ by $0\leq u_{n}\leq 1$ for any $n\in\mathbb{N}$, the
Banach--Saks theorem \cite[Theorem A.4.1-(i)]{CF} yields $h_{0}\in L^{2}(K,\mu)$ and a strictly
increasing sequence $\{j_{k}\}_{k=1}^{\infty}\subset\mathbb{N}$ such that the Ces\`{a}ro
mean sequence $\{\overline{u}_{n}\}_{n=1}^{\infty}\subset\mathcal{F}\cap\mathcal{C}(K)$
of $\{u_{j_{k}}\}_{k=1}^{\infty}$ given by $\overline{u}_{n}:=n^{-1}\sum_{k=1}^{n}u_{j_{k}}$
satisfies $\lim_{n\to\infty}\|h_{0}-\overline{u}_{n}\|_{2}=0$.
Then \ref{it:un} obviously holds for $\{\overline{u}_{n}\}_{n=1}^{\infty}$,
and it follows from \eqref{eq:CapAlphaV00V01} and \eqref{eq:a0-approx-core-sym}
that $\lim_{n\to\infty}\mathcal{E}(\overline{u}_{n},\overline{u}_{n})=a_{0}$
and that for any $n,k\in\mathbb{N}$,
\begin{align*}
\mathcal{E}(\overline{u}_{n}-\overline{u}_{k},\overline{u}_{n}-\overline{u}_{k})
	&=2\mathcal{E}(\overline{u}_{n},\overline{u}_{n})+2\mathcal{E}(\overline{u}_{k},\overline{u}_{k})
		-4\mathcal{E}((\overline{u}_{n}+\overline{u}_{k})/2,(\overline{u}_{n}+\overline{u}_{k})/2)\\
&\leq 2\mathcal{E}(\overline{u}_{n},\overline{u}_{n})+2\mathcal{E}(\overline{u}_{k},\overline{u}_{k})-4a_{0}
	\xrightarrow{n\wedge k\to\infty}0,
\end{align*}
which together with $\lim_{n\to\infty}\|h_{0}-\overline{u}_{n}\|_{2}=0$ and
the completeness of $(\mathcal{F},\mathcal{E}_{1})$ implies that $h_{0}\in\mathcal{F}$ and
$\lim_{n\to\infty}\mathcal{E}_{1}(h_{0}-\overline{u}_{n},h_{0}-\overline{u}_{n})=0$. Thus
$h_{0}\in\overline{\varphi_{0}+\mathcal{C}_{U_{0}}}^{\mathcal{F}}
	=\varphi_{0}+\mathcal{F}_{U_{0}}=h_{0}+\mathcal{F}_{U_{0}}$
by $\{\overline{u}_{n}\}_{n=1}^{\infty}\subset\varphi_{0}+\mathcal{C}_{U_{0}}$,
$\mathcal{E}(h_{0},h_{0})=\lim_{n\to\infty}\mathcal{E}(\overline{u}_{n},\overline{u}_{n})=a_{0}$,
and therefore $h_{0}$ is $\mathcal{E}$-harmonic on $U_{0}$
in view of \eqref{eq:CapAlphaV00V01} and \eqref{eq:harmonic-var}, completing the proof.
\end{proof}
%
%%%%%
\begin{figure}[t]\centering
	\begin{minipage}{.49\linewidth}\centering%
		\hspace*{-3pt}\includegraphics[height=200pt]{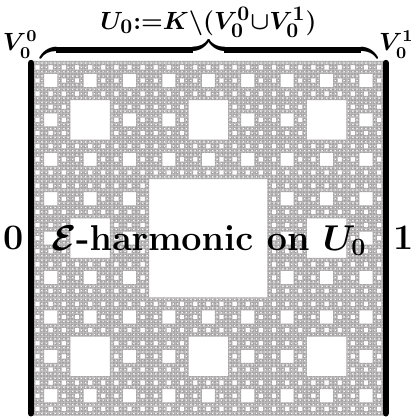}%
		\caption{The choice of $h_{0}\in\mathcal{F}$}%
		\label{fig:h0U0-dfn}%
	\end{minipage}%
	\hspace*{6pt}%
	\begin{minipage}{.495\linewidth}\centering%
		\includegraphics[height=200pt]{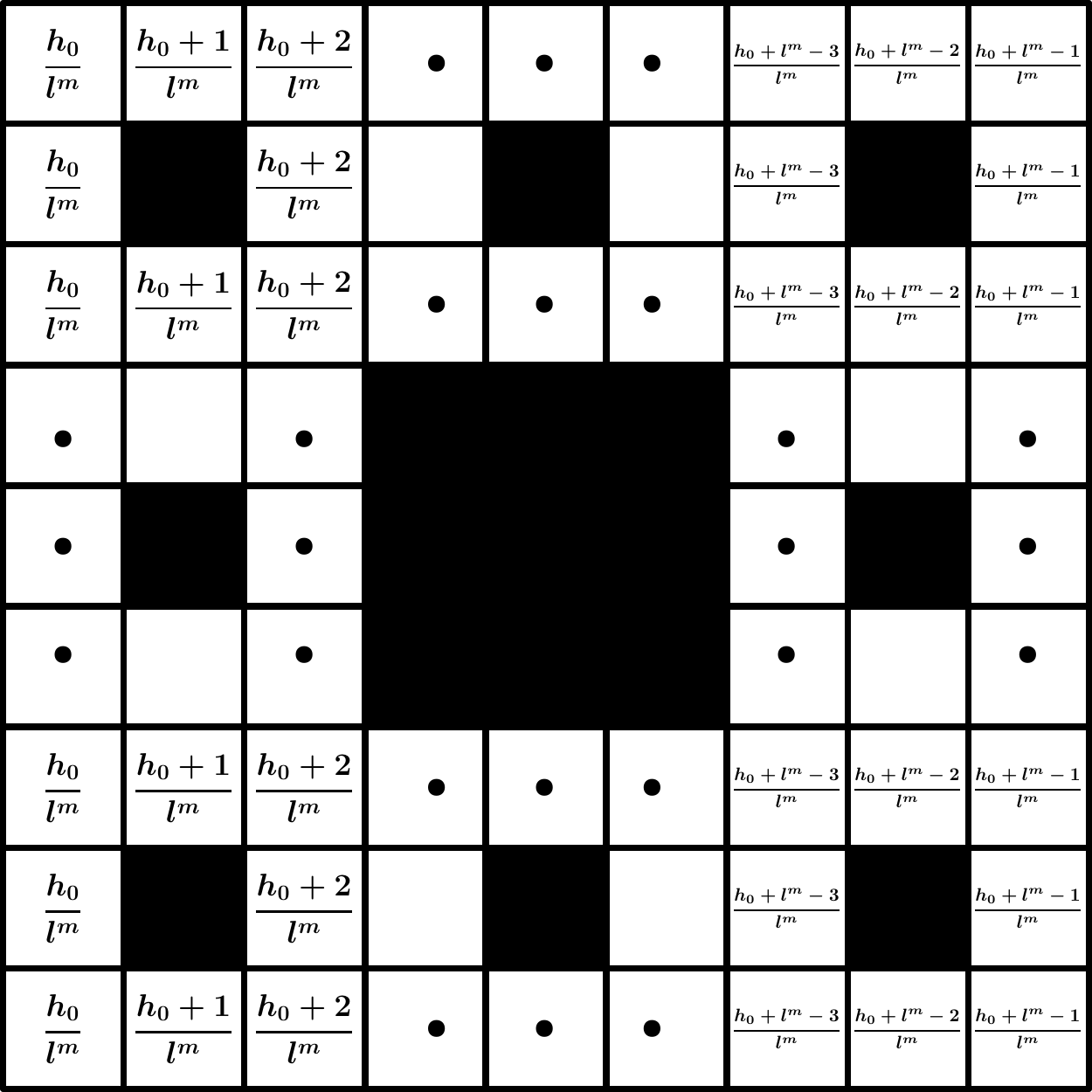}%
		\caption{The construction of $h_{m}\in h_{0}+\mathcal{F}_{U_{0}}$}%
		\label{fig:hm-dfn}%
	\end{minipage}%
\end{figure}
%%%%%
We need the following two lemmas for the remaining two propositions and their proofs.
\begin{lemma}\label{lem:hm-domain-boundary-value}
Let $h_{0}\in\mathcal{F}$ be as in Proposition \textup{\ref{prop:h0-approx}},
let $m\in\mathbb{N}$ and define $h_{m}\in L^{2}(K,\mu)$ by
\begin{equation}\label{eq:hm-dfn}
h_{m}:=\sum_{w\in W_{m}}(F_{w})_{*}(l^{-m}h_{0}+q^{w}_{1}\one_{K})
\end{equation}
\textup{(see Figure \ref{fig:hm-dfn} below for an illustration of \eqref{eq:hm-dfn})}.
Then $h_{m}\in h_{0}+\mathcal{F}_{U_{0}}$.
\end{lemma}
\begin{proof}
Let $\{u_{n}\}_{n=1}^{\infty}\subset\mathcal{F}\cap\mathcal{C}(K)$ be as in Proposition \ref{prop:h0-approx}.
For each $n\in\mathbb{N}$, since $u_{n}\circ g=u_{n}$ for any $g\in\mathcal{G}_{1}$,
$u_{n}\in\varphi_{0}+\mathcal{C}_{U_{0}}$ and hence
$u_{n}|_{V_{0}^{0}\cup V_{0}^{1}}=\varphi_{0}|_{V_{0}^{0}\cup V_{0}^{1}}=\one_{V_{0}^{1}}$
by Proposition \ref{prop:h0-approx}-\ref{it:un}, we can define $u_{m,n}\in\mathcal{C}(K)$
by setting $u_{m,n}|_{K_{w}}:=(l^{-m}u_{n}+q^{w}_{1}\one_{K})\circ F_{w}^{-1}$ for each $w\in W_{m}$,
so that $u_{m,n}\circ F_{w}=l^{-m}u_{n}+q^{w}_{1}\one_{K}\in\mathcal{F}$ by $\one_{K}\in\mathcal{F}$
and thus $u_{m,n}\in\varphi_{0}+\mathcal{C}_{U_{0}}$ by \ref{GSCDF2} and $u_{n}\in\varphi_{0}+\mathcal{C}_{U_{0}}$.
Then we see from \ref{GSCDF3}, Lemmas \ref{lem:cell-intersection-null}, \ref{lem:Fw-star}
and Proposition \ref{prop:h0-approx}-\ref{it:h0-approx} that $\{u_{m,n}\}_{n=1}^{\infty}$
is a Cauchy sequence in the Hilbert space $(\mathcal{F},\mathcal{E}_{1})$ with
$\lim_{n\to\infty}\|h_{m}-u_{m,n}\|_{2}=0$ and therefore has to converge to
$h_{m}$ in norm in $(\mathcal{F},\mathcal{E}_{1})$, whence
$h_{m}\in\overline{\varphi_{0}+\mathcal{C}_{U_{0}}}^{\mathcal{F}}
	=\varphi_{0}+\mathcal{F}_{U_{0}}=h_{0}+\mathcal{F}_{U_{0}}$
by $\{u_{m,n}\}_{n=1}^{\infty}\subset\varphi_{0}+\mathcal{C}_{U_{0}}$
and Proposition \ref{prop:h0-approx}-\ref{it:h0}.
\end{proof}
\begin{lemma}\label{lem:coordinate-func}
Let $k\in\{1,2,\ldots,d\}$ and define $f_{k}\in\mathcal{C}(\mathbb{R}^{d})$ by
$f_{k}((x_{j})_{j=1}^{d}):=x_{k}$. Then either $f_{k}|_{K}\in\mathcal{F}$ and
$\mathcal{E}(f_{k}|_{K},f_{k}|_{K})>0$ or $f_{k}|_{K}\not\in\mathcal{F}$.
\end{lemma}
\begin{proof}
Suppose to the contrary that $f_{k}|_{K}\in\mathcal{F}$ and $\mathcal{E}(f_{k}|_{K},f_{k}|_{K})=0$.
Then $f|_{K}\in\mathcal{F}$ and $\mathcal{E}(f|_{K},f|_{K})=0$ for any
$f\in\{\one_{\mathbb{R}^{d}},f_{1},f_{2},\ldots,f_{d}\}$ by $\one_{K}\in\mathcal{F}$,
$\mathcal{E}(\one_{K},\one_{K})=0$ and \ref{GSCDF1}, and hence also for any polynomial $f\in\mathcal{C}(\mathbb{R}^{d})$
since for any $u,v\in\mathcal{F}\cap\mathcal{C}(K)$ we have $uv\in\mathcal{F}\cap\mathcal{C}(K)$ and
$\mathcal{E}(uv,uv)^{1/2}\leq\|u\|_{\sup}\mathcal{E}(v,v)^{1/2}+\|v\|_{\sup}\mathcal{E}(u,u)^{1/2}$
by \cite[Theorem 1.4.2-(ii)]{FOT}. On the other hand, $\mathcal{E}(u,u)>0$ for some
$u\in\mathcal{F}\cap\mathcal{C}(K)$ by the existence of such $u\in\mathcal{F}$
and the denseness of $\mathcal{F}\cap\mathcal{C}(K)$ in $(\mathcal{F},\mathcal{E}_{1})$,
the Stone--Weierstrass theorem \cite[Theorem 2.4.11]{Dud} implies that
$\lim_{n\to\infty}\|u-f_{n}|_{K}\|_{\sup}=0$ for some sequence
$\{f_{n}\}_{n=1}^{\infty}\subset\mathcal{C}(\mathbb{R}^{d})$ of polynomials, then
$\lim_{n\to\infty}\|u-f_{n}|_{K}\|_{2}=0$ and
$\lim_{n\wedge k\to\infty}\mathcal{E}_{1}(f_{n}|_{K}-f_{k}|_{K},f_{n}|_{K}-f_{k}|_{K})=0$.
Thus $\lim_{n\to\infty}\mathcal{E}_{1}(u-f_{n}|_{K},u-f_{n}|_{K})=0$ by the completeness
of $(\mathcal{F},\mathcal{E}_{1})$ and therefore
$0<\mathcal{E}(u,u)=\lim_{n\to\infty}\mathcal{E}(f_{n}|_{K},f_{n}|_{K})=0$,
which is a contradiction and completes the proof.
\end{proof}
\begin{proposition}\label{prop:h0-nonzero}
Let $h_{0}\in\mathcal{F}$ be as in Proposition \textup{\ref{prop:h0-approx}}.
Then $\mathcal{E}(h_{0},h_{0})>0$.
\end{proposition}
\begin{proof}
Let $f_{1}\in\mathcal{C}(\mathbb{R}^{d})$ be as in Lemma \ref{lem:coordinate-func} with $k=1$ and
for each $m\in\mathbb{N}$ let $h_{m}\in\mathcal{F}$ be as in Lemma \ref{lem:hm-domain-boundary-value},
so that by \eqref{eq:hm-dfn}, Lemmas \ref{lem:cell-intersection-null} and \ref{lem:Fw-star} we have
$\|f_{1}|_{K}-h_{m}\|_{2}=l^{-m}\|f_{1}|_{K}-h_{0}\|_{2}$
and $h_{m}\circ F_{w}=l^{-m}h_{0}+q^{w}_{1}\one_{K}$ $\mu$-a.e.\ for any $w\in W_{m}$
and hence \eqref{eq:GSCDF3} for $u\in\mathcal{F}$ from Lemma \ref{lem:Fw-star} and
Lemma \ref{lem:conservative} together yield
\begin{equation}\label{eq:hm-GSCDF3}
\mathcal{E}(h_{m},h_{m})
	=\sum_{w\in W_{m}}\frac{1}{r^{m}}\mathcal{E}(l^{-m}h_{0}+q^{w}_{1}\one_{K},l^{-m}h_{0}+q^{w}_{1}\one_{K})
	=\Bigl(\frac{\#S}{r}l^{-2}\Bigr)^{m}\mathcal{E}(h_{0},h_{0}).
\end{equation}

Now if $\mathcal{E}(h_{0},h_{0})=0$, then $\mathcal{E}(h_{m},h_{m})=0$ by \eqref{eq:hm-GSCDF3}
for any $m\in\mathbb{N}$, thus $\{h_{m}\}_{m=1}^{\infty}$ would be
a Cauchy sequence in the Hilbert space $(\mathcal{F},\mathcal{E}_{1})$ with
$\lim_{m\to\infty}\|f_{1}|_{K}-h_{m}\|_{2}=0$ and therefore convergent to
$f_{1}|_{K}$ in norm in $(\mathcal{F},\mathcal{E}_{1})$, hence $f_{1}|_{K}\in\mathcal{F}$
and $\mathcal{E}(f_{1}|_{K},f_{1}|_{K})=\lim_{m\to\infty}\mathcal{E}(h_{m},h_{m})=0$,
which contradicts Lemma \ref{lem:coordinate-func} and completes the proof.
\end{proof}
It is the proof of the following proposition that requires our standing assumption that
$S\not=\{0,1,\ldots,l-1\}^{d}$, which excludes the case of $K=[0,1]^{d}$ from the present framework.
\begin{proposition}\label{prop:h2-non-harmonic}
Let $h_{2}\in\mathcal{F}$ be as in Lemma \textup{\ref{lem:hm-domain-boundary-value}}
with $m=2$. Then $h_{2}$ is not $\mathcal{E}$-harmonic on $U_{0}$.
\end{proposition}
\begin{proof}
We claim that, if $h_{2}$ were $\mathcal{E}$-harmonic on $U_{0}$,
then $h_{0}\in\mathcal{F}$ as in Proposition \ref{prop:h0-approx} would
turn out to be $\mathcal{E}$-harmonic on $K\setminus V_{0}^{0}$, which would imply that
$\mathcal{E}(h_{0},h_{0})=\lim_{n\to\infty}\mathcal{E}(h_{0},u_{n})=0$ for
$\{u_{n}\}_{n=1}^{\infty}\subset\varphi_{0}+\mathcal{C}_{U_{0}}\subset\mathcal{C}_{K\setminus V_{0}^{0}}$
as in Proposition \ref{prop:h0-approx} by \eqref{eq:harmonic-weak},
a contradiction to Proposition \ref{prop:h0-nonzero} and will thereby
prove that $h_{2}$ is not $\mathcal{E}$-harmonic on $U_{0}$.

For each $\varepsilon=(\varepsilon_{k})_{k=1}^{d}\in\{1\}\times\{0,1\}^{d-1}$,
set $U^{\varepsilon}:=K\cap\prod_{k=1}^{d}(\varepsilon_{k}-1,\varepsilon_{k}+1)$,
$K^{\varepsilon}:=K\cap\prod_{k=1}^{d}[\varepsilon_{k}-1/2,\varepsilon_{k}+1/2]$
and choose $\varphi_{\varepsilon}\in\mathcal{C}_{U^{\varepsilon}}$ so that
$\varphi_{\varepsilon}|_{K^{\varepsilon}}=\one_{K^{\varepsilon}}$;
such $\varphi_{\varepsilon}$ exists by \cite[Exercise 1.4.1]{FOT}. Let
$v\in\mathcal{C}_{K\setminus V_{0}^{0}}$ and, taking an enumeration $\{\varepsilon^{(k)}\}_{k=1}^{2^{d-1}}$
of $\{1\}\times\{0,1\}^{d-1}$ and recalling that $v_{1}v_{2}\in\mathcal{F}\cap\mathcal{C}(K)$
for any $v_{1},v_{2}\in\mathcal{F}\cap\mathcal{C}(K)$ by \cite[Theorem 1.4.2-(ii)]{FOT}, define
$v_{\varepsilon}\in\mathcal{C}_{U^{\varepsilon}}$ for $\varepsilon\in\{1\}\times\{0,1\}^{d-1}$
by $v_{\varepsilon^{(1)}}:=v\varphi_{\varepsilon^{(1)}}$ and
$v_{\varepsilon^{(k)}}:=v\varphi_{\varepsilon^{(k)}}\prod_{j=1}^{k-1}(\one_{K}-\varphi_{\varepsilon^{(j)}})$
for $k\in\{2,\ldots,2^{d-1}\}$. Then
$v-\sum_{\varepsilon\in\{1\}\times\{0,1\}^{d-1}}v_{\varepsilon}
	=v\prod_{\varepsilon\in\{1\}\times\{0,1\}^{d-1}}(\one_{K}-\varphi_{\varepsilon})
	\in\mathcal{C}_{U_{0}}$,
hence $\mathcal{E}(h_{0},v)=\sum_{\varepsilon\in\{1\}\times\{0,1\}^{d-1}}\mathcal{E}(h_{0},v_{\varepsilon})$
by Proposition \ref{prop:h0-approx}-\ref{it:h0} and \eqref{eq:harmonic-weak}, and therefore
the desired $\mathcal{E}$-harmonicity of $h_{0}$ on $K\setminus V_{0}^{0}$, i.e.,
\eqref{eq:harmonic-weak} with $h=h_{0}$ and $U=K\setminus V_{0}^{0}$, would be
obtained by deducing that $\mathcal{E}(h_{0},v_{\varepsilon})=0$ for any
$\varepsilon\in\{1\}\times\{0,1\}^{d-1}$.

To this end, set $\varepsilon^{(0)}:=(\one_{\{1\}}(k))_{k=1}^{d}$, take
$i=(i_{k})_{k=1}^{d}\in S$ with $i_{1}<l-1$ and $i+\varepsilon^{(0)}\not\in S$,
which exists by $\emptyset\not=S\subsetneq\{0,1,\ldots,l-1\}^{d}$ and \ref{GSC1},
and let $\varepsilon=(\varepsilon_{k})_{k=1}^{d}\in\{1\}\times\{0,1\}^{d-1}$. We will choose
$i^{\varepsilon}\in S$ with $F_{ii^{\varepsilon}}(\varepsilon)\in F_{i}(K\cap(\{1\}\times(0,1)^{d-1}))$
and assemble $v_{\varepsilon}\circ g_{w}\circ F_{w}^{-1}$ with a suitable
$g_{w}\in\mathcal{G}_{1}$ for $w\in W_{2}$ with $F_{ii^{\varepsilon}}(\varepsilon)\in K_{w}$
into a function $v_{\varepsilon,2}\in\mathcal{C}_{U_{0}}$. Specifically, set
$i^{\varepsilon,\eta}:=\bigl((l-1)(\one_{\{1\}}(k)+1-\varepsilon_{k})+(2\varepsilon_{k}-1)\eta_{k}\bigr)_{k=1}^{d}$
for each $\eta=(\eta_{k})_{k=1}^{d}\in\{0\}\times\{0,1\}^{d-1}$ and
$I^{\varepsilon}:=\{\eta\in\{0\}\times\{0,1\}^{d-1}\mid i^{\varepsilon,\eta}\in S\}$,
so that $i^{\varepsilon}:=i^{\varepsilon,\zero_{d}}\in S$ by \ref{GSC4} and \ref{GSC1} and hence
$\zero_{d}\in I^{\varepsilon}$. Thanks to $v_{\varepsilon}\in\mathcal{C}_{U^{\varepsilon}}$ and
$i+\varepsilon^{(0)}\not\in S$ we can define $v_{\varepsilon,2}\in\mathcal{C}(K)$ by setting
\begin{equation}\label{eq:vepsilon2}
v_{\varepsilon,2}|_{K_{w}}:=
	\begin{cases}
	v_{\varepsilon}\circ g_{\eta}\circ F_{w}^{-1}
		& \textrm{if $\eta\in I^{\varepsilon}$ and $w=ii^{\varepsilon,\eta}$}\\
	0 & \textrm{if $w\not\in\{ii^{\varepsilon,\eta}\mid\eta\in I^{\varepsilon}\}$}
	\end{cases}
	\quad\textrm{for each $w\in W_{2}$,}
\end{equation}
then $\supp_{K}[v_{\varepsilon,2}]\subset K_{i}\setminus V_{0}^{0}\subset U_{0}$
by \eqref{eq:vepsilon2} and $i_{1}<l-1$, $v_{\varepsilon,2}\circ F_{w}\in\mathcal{F}$ for any
$w\in W_{2}$ by \eqref{eq:vepsilon2}, $v_{\varepsilon}\in\mathcal{F}\cap\mathcal{C}(K)$
and \ref{GSCDF1}, thus $v_{\varepsilon,2}\in\mathcal{F}$ by \ref{GSCDF2} and therefore
$v_{\varepsilon,2}\in\mathcal{C}_{U_{0}}$. Moreover, recalling that
$h_{2}\circ F_{w}=l^{-2}h_{0}+q^{w}_{1}\one_{K}$ $\mu$-a.e.\ for any $w\in W_{2}$ by
\eqref{eq:hm-dfn}, Lemmas \ref{lem:cell-intersection-null} and \ref{lem:Fw-star} and
letting $\{u_{n}\}_{n=1}^{\infty}\subset\mathcal{F}\cap\mathcal{C}(K)$ be as in
Proposition \ref{prop:h0-approx}, we see from \eqref{eq:GSCDF3} for $u\in\mathcal{F}$
in Lemma \ref{lem:Fw-star}, \eqref{eq:vepsilon2}, Lemma \ref{lem:conservative},
Proposition \ref{prop:h0-approx}-\ref{it:h0-approx}, \ref{GSCDF1} and
Proposition \ref{prop:h0-approx}-\ref{it:un} that
\begin{equation}\label{eq:Eh2vepsilon2}
\begin{split}
\mathcal{E}(h_{2},v_{\varepsilon,2})
	&=\sum_{\eta\in I^{\varepsilon}}\frac{1}{r^{2}l^{2}}\mathcal{E}(h_{0},v_{\varepsilon}\circ g_{\eta})
	=\lim_{n\to\infty}\sum_{\eta\in I^{\varepsilon}}\frac{1}{r^{2}l^{2}}\mathcal{E}(u_{n},v_{\varepsilon}\circ g_{\eta})\\
&=\lim_{n\to\infty}\sum_{\eta\in I^{\varepsilon}}\frac{1}{r^{2}l^{2}}\mathcal{E}(u_{n}\circ g_{\eta},v_{\varepsilon})
	=\lim_{n\to\infty}\frac{\#I^{\varepsilon}}{r^{2}l^{2}}\mathcal{E}(u_{n},v_{\varepsilon})
	=\frac{\#I^{\varepsilon}}{r^{2}l^{2}}\mathcal{E}(h_{0},v_{\varepsilon}).
\end{split}
\end{equation}
Now, supposing that $h_{2}$ were $\mathcal{E}$-harmonic on $U_{0}$, from \eqref{eq:Eh2vepsilon2},
$\#I^{\varepsilon}>0$, $v_{\varepsilon,2}\in\mathcal{C}_{U_{0}}$ and \eqref{eq:harmonic-weak} we would obtain
$\mathcal{E}(h_{0},v_{\varepsilon})=r^{2}l^{2}(\#I^{\varepsilon})^{-1}\mathcal{E}(h_{2},v_{\varepsilon,2})=0$,
which would imply a contradiction as explained in the last two paragraphs and thus completes the proof.
\end{proof}
\begin{proof}[of Theorem \textup{\ref{thm:dwSC}}]
Let $h_{0}\in\mathcal{F}$ be as in Proposition \ref{prop:h0-approx} and let
$h_{2}\in h_{0}+\mathcal{F}_{U_{0}}$ be as in Lemma \ref{lem:hm-domain-boundary-value} with $m=2$,
so that $h_{0}$ is $\mathcal{E}$-harmonic on $U_{0}$, $h_{2}+\mathcal{F}_{U_{0}}=h_{0}+\mathcal{F}_{U_{0}}$,
$h_{2}$ is not $\mathcal{E}$-harmonic on $U_{0}$ by Proposition \ref{prop:h2-non-harmonic}
and hence $\mathcal{E}(h_{0},h_{0})<\mathcal{E}(h_{2},h_{2})$ in view of
\eqref{eq:harmonic-var}. This strict inequality combined with \eqref{eq:hm-GSCDF3} shows that
\begin{equation*}
\mathcal{E}(h_{0},h_{0})<\mathcal{E}(h_{2},h_{2})=\Bigl(\frac{\#S}{r}l^{-2}\Bigr)^{2}\mathcal{E}(h_{0},h_{0}),
\end{equation*}
whence $l^{2}<\#S/r$, namely $d_{\mathrm{w}}=\log_{l}(\#S/r)>2$.
\end{proof}
\begin{remark}\label{rmk:h2-non-harmonic-maximum-principle}
As an alternative to Proposition \ref{prop:h2-non-harmonic} and its proof above,
we could have used the \emph{strong maximum principle} for $\mathcal{E}$-harmonic functions
to show that $h_{1}\in\mathcal{F}$ as in Lemma \ref{lem:hm-domain-boundary-value} with $m=1$
is not $\mathcal{E}$-harmonic on $U_{0}$, from which Theorem \ref{thm:dwSC} follows in exactly
the same way. Indeed, let $h_{0}\in\mathcal{F}$ be as in Proposition \ref{prop:h0-approx},
let $i=(i_{k})_{k=1}^{d}\in S$ be as in the above proof of Proposition \ref{prop:h2-non-harmonic}
and set $U_{i}:=F_{i}(K\cap((0,1]\times(0,1)^{d-1}))$, which is an open subset of $U_{0}$ by $i_{1}<l-1$
and $i+(\one_{\{1\}}(k))_{k=1}^{d}\not\in S$ and connected by \ref{GSC4}, \ref{GSC1} and \ref{GSC2}.
Then noting that $h_{1}|_{U_{i}}=l^{-1}h_{0}\circ F_{i}^{-1}|_{U_{i}}+l^{-1}i_{1}\one_{U_{i}}$ $\mu$-a.e.\ by
\eqref{eq:hm-dfn} and that $\mu(K_{i}\setminus U_{i})=0$ by Lemma \ref{lem:cell-intersection-null},
we easily see from Proposition \ref{prop:h0-approx} (and its proof), Lemma \ref{lem:Fw-star},
Proposition \ref{prop:h0-nonzero} and $\mathcal{E}(\one_{K},\one_{K})=0$ that
$h_{1}\leq l^{-1}(i_{1}+1)$ $\mu$-a.e.\ on $U_{i}$, ``$h_{1}=l^{-1}(i_{1}+1)$ on $U_{i}\cap F_{i}(V_{0}^{1})$'' and
$h_{1}|_{U_{i}}\not=l^{-1}(i_{1}+1)\one_{U_{i}}$, so that $h_{1}$ cannot be $\mathcal{E}$-harmonic on $U_{i}$
or on $U_{0}\supset U_{i}$ by the strong maximum principle \cite[Theorem 2.11]{CK}
for $\mathcal{E}$-(sub)harmonic functions on $U_{i}$.

While this short ``proof'' nicely illustrates the heuristics behind the proof of
Proposition \ref{prop:h2-non-harmonic} above, it is in fact highly non-trivial to justify
our last application of the strong maximum principle \cite[Theorem 2.11]{CK} because it
requires the following set of conditions; see \cite[Sections A.2, 1.4, 4.1, 4.2 and 1.6]{FOT}
for the definitions and further details of the notions involved.
Note that \hyperlink{con}{\textup{(con)}} below also gives a precise formulation of the statement
``$h_{1}=l^{-1}(i_{1}+1)$ on $U_{i}\cap F_{i}(V_{0}^{1})$'' in the previous paragraph,
which needs to be done since $\mu(U_{i}\cap F_{i}(V_{0}^{1}))=0$ by Lemma \ref{lem:cell-intersection-null}.
\begin{enumerate}[label=\textup{(con)},align=left,leftmargin=*,topsep=4pt,parsep=0pt,itemsep=2pt]
\item[\hypertarget{AC}{\textup{(AC)}}]There exists a $\mu$-symmetric Hunt process
	$X=\bigl(\Omega,\mathscr{M},\{X_{t}\}_{t\in[0,\infty]},\{\mathbb{P}_{x}\}_{x\in K_{\Delta}}\bigr)$
	on $K$ such that its Dirichlet form on $L^{2}(K,\mu)$ is $(\mathcal{E},\mathcal{F})$
	and the Borel measure $\mathbb{P}_{x}[X_{t}\in\cdot]$ on $K$ is
	\emph{absolutely continuous with respect to $\mu$ for any $(t,x)\in(0,\infty)\times K$},
	where $K_{\Delta}:=K\cup\{\Delta\}$ denotes the one-point compactification of $K$.
\item[\hypertarget{irr}{\textup{(irr)}}]
	$(\mathcal{E}|_{\mathcal{F}_{U_{i}}\times\mathcal{F}_{U_{i}}},\mathcal{F}_{U_{i}})$
	is \emph{$\mu$-irreducible}, i.e., $\mu(A)\mu(U_{i}\setminus A)=0$ for any Borel subset
	$A$ of $U_{i}$ with the property that $u\one_{A}\in\mathcal{F}_{U_{i}}$ and
	$\mathcal{E}(u\one_{A},u-u\one_{A})=0$ for any $u\in\mathcal{F}_{U_{i}}$.
\item[\hypertarget{con}{\textup{(con)}}]If $h_{1}$ were $\mathcal{E}$-harmonic on $U_{i}$,
	then there would exist a $\mu$-version of $h_{1}|_{U_{i}}$ which would be nearly Borel
	measurable and finely continuous with respect to $X$ as in \hyperlink{AC}{\textup{(AC)}}
	and satisfy $h_{1}(x)=l^{-1}(i_{1}+1)$ for some $x\in U_{i}\cap F_{i}(V_{0}^{1})$.
\end{enumerate}
(Some other versions of the strong maximum principle for (sub)harmonic functions in the setting
of a strongly local regular symmetric Dirichlet form have been proved in \cite{Kuw00,Kuw08,Kuw12},
but conditions very similar to \hyperlink{AC}{\textup{(AC)}}, \hyperlink{irr}{\textup{(irr)}} and
\hyperlink{con}{\textup{(con)}} are assumed also by them and therefore have to be verified anyway.)

It is possible to deduce \hyperlink{AC}{\textup{(AC)}}, \hyperlink{irr}{\textup{(irr)}} and
\hyperlink{con}{\textup{(con)}} from known properties of $(K,\rho,\mu,\mathcal{E},\mathcal{F})$.
Indeed, \hyperlink{AC}{\textup{(AC)}} follows from \cite[Chapter I, Theorem 9.4]{BG} and
the fact that $(K,\rho,\mu,\mathcal{E},\mathcal{F})$ has a version $p=p_{t}(x,y)$ of the heat
kernel which is jointly continuous in $(t,x,y)\in(0,\infty)\times K\times K$ and satisfies
\eqref{eq:HKEdw} by \cite[Theorem 4.30 and Remark 4.33]{BBKT} and \cite[Theorem 3.1]{BGK},
and \hyperlink{irr}{\textup{(irr)}} is implied by \textup{w-LLE($(\cdot)^{d_{\mathrm{w}}}$)}
from \cite[Theorem 3.1]{BGK}, the connectedness of $U_{i}$ and \cite[Theorem 1.6.1]{FOT}.
For \hyperlink{con}{\textup{(con)}}, let $\Capa^{\mathcal{E}}_{1}(A)$ denote the $1$-capacity
of $A\subset K$ with respect to $(K,\mu,\mathcal{E},\mathcal{F})$ as defined in
\cite[(2.1.1), (2.1.2) and (2.1.3)]{FOT}, and suppose that $h_{1}$ were $\mathcal{E}$-harmonic on $U_{i}$.
Then by \textup{w-LLE($(\cdot)^{d_{\mathrm{w}}}$)} and \cite[Corollary 4.2]{BGK}
there would exist a continuous $\mu$-version of $h_{1}|_{U_{i}}$ and hence so would
a continuous one of $h_{0}=lh_{1}\circ F_{i}-i_{1}\one_{K}$ on $F_{i}^{-1}(U_{i})$
by Lemma \ref{lem:Fw-star}. Moreover, $\Capa^{\mathcal{E}}_{1}(V_{0}^{1})>0$
by Propositions \ref{prop:h0-approx}-\ref{it:h0}, \ref{prop:h0-nonzero}, \eqref{eq:harmonic-var}
and \cite[Corollary 2.3.1]{FOT}, $\Capa^{\mathcal{E}}_{1}(F_{i}^{-1}(U_{i})\cap V_{0}^{1})>0$
by \ref{GSC4}, \ref{GSC1} and \cite[Lemma 7.14]{K:cdsa}, the continuous
$\mu$-version of $h_{0}$ on $F_{i}^{-1}(U_{i})$ would satisfy $h_{0}=1$ on
$F_{i}^{-1}(U_{i})\cap V_{0}^{1}\setminus N$ for some $N\subset K$ with $\Capa^{\mathcal{E}}_{1}(N)=0$
by $h_{0}\in\varphi_{0}+\mathcal{F}_{U_{0}}$ and \cite[Theorem 2.1.4-(i)]{FOT}, and
$F_{i}^{-1}(U_{i})\cap V_{0}^{1}\setminus N\not=\emptyset$ by
$\Capa^{\mathcal{E}}_{1}(F_{i}^{-1}(U_{i})\cap V_{0}^{1})>\Capa^{\mathcal{E}}_{1}(N)$.
Thus $h_{0}(y)=1$ for some $y\in F_{i}^{-1}(U_{i})\cap V_{0}^{1}$ and $h_{1}(x)=l^{-1}(i_{1}+1)$
for $x:=F_{i}(y)\in U_{i}\cap F_{i}(V_{0}^{1})$, proving \hyperlink{con}{\textup{(con)}}.

Note that the arguments in the last paragraph heavily rely on the demanding results in \cite{BBKT,BGK},
and also that the proofs of the strong maximum principle in \cite{Kuw00,Kuw08,Kuw12,CK} make
full use of the potential theory for regular symmetric Dirichlet forms in \cite{FOT,CF}.
In this sense, the above ``short'' proof of the non-$\mathcal{E}$-harmonicity of $h_{1}$
on $U_{0}$ based on the strong maximum principle is far from self-contained and much less
elementary than the proof of Proposition \ref{prop:h2-non-harmonic}.
\end{remark}
\begin{acknowledgements}
The author would like to thank Takashi Kumagai and Mathav Murugan for their valuable
comments on an earlier version of this paper.
\end{acknowledgements}
%
%%%

%%%
\affiliationone{% in this example, two authors share an institution
   Naotaka Kajino\\
   Department of Mathematics, Graduate School of Science, Kobe University\\
   Rokkodai-cho 1-1, Nada-ku\\
   Kobe 657-8501\\
   Japan
   \email{nkajino@math.kobe-u.ac.jp}}
% Important: Do not put any empty line here.
\affiliationtwo{ } %inserts a space to make this field empty
\affiliationthree{%
   Current address:\\
   Research Institute for Mathematical Sciences, Kyoto University\\
   Kitashirakawa-Oiwake-cho, Sakyo-ku\\
   Kyoto 606-8502\\
   Japan
   \email{nkajino@kurims.kyoto-u.ac.jp}}
\end{document}